\documentclass[10pt]{amsart}

\usepackage{amssymb,latexsym, mathtools,tikz}
\usepackage{enumerate}
\usepackage{graphicx}
\usepackage{float}
\usepackage{placeins}
\usepackage{mdframed}
\usepackage{amssymb}
\usepackage{esint}
\usepackage{cool}
\usepackage[all,cmtip]{xy}
\usepackage{mathtools}
\usepackage{amstext} 
\usepackage{array}   
\usepackage[shortlabels]{enumitem}
\usepackage{ytableau}

\newcolumntype{L}{>{$}l<{$}} 
\newcolumntype{C}{>{$}c<{$}}
\newtheorem{theorem}{Theorem}[section]
\newtheorem{lemma}[theorem]{Lemma}
\newtheorem{cor}[theorem]{Corollary}
\newtheorem{prop}[theorem]{Proposition}
\newtheorem{setup}[theorem]{Setup}
\theoremstyle{definition}
\newtheorem{definition}[theorem]{Definition}
\newtheorem{example}[theorem]{Example}

\newtheorem{obs}[theorem]{Observation}
\newtheorem{notation}[theorem]{Notation}

\theoremstyle{remark}
\newtheorem{remark}[theorem]{Remark}

\newtheorem{the context}[theorem]{The Context}
\newtheorem{question}[theorem]{Question}
\numberwithin{equation}{theorem}
\numberwithin{equation}{section}

\newcommand{\cat}[1]{\mathcal{#1}}

\newcommand{\tor}{\operatorname{Tor}}

\newcommand{\proj}{\operatorname{Proj}}

\newcommand{\Ker}{\operatorname{Ker}}

\newcommand{\ideal}[1]{\mathfrak{#1}}
\newcommand{\m}{\ideal{m}}

\newcommand{\supp}{\operatorname{Supp}}

\newcommand{\bbz}{\mathbb{Z}}
\newcommand{\bbn}{\mathbb{N}}

\renewcommand{\geq}{\geqslant}
\renewcommand{\leq}{\leqslant}
\renewcommand{\ker}{\Ker}

\newcommand{\st}{\operatorname{Star}}

\newcommand{\maps}[5]{\xymatrix{#1 \ar[r]^-{#3} & #2 \\
#4 \ar@{|->}[r] & #5 \\}}

\newcommand{\mfa}{\mathfrak{a}}

\newcommand{\lcm}{\textrm{lcm}}

\def\w{\wedge}
\def\proj{\operatorname{proj}}

\newcommand{\dstar}{\operatorname{**}}
\newcommand{\sq}{\operatorname{sq}}

\newcommand{\po}{\operatorname{po}}

\begin{document}
\title{On Constructions Related to the Generalized Taylor Complex}

\author{Keller VandeBogert }
\date{\today}

\maketitle

\begin{abstract}
    In this paper, we extend constructions and results for the Taylor complex to the generalized Taylor complex constructed by Herzog. We construct an explicit DG-algebra structure on the generalized Taylor complex and extend a result of Katth\"an on quotients of the Taylor complex by DG-ideals. We introduce a generalization of the Scarf complex for families of monomial ideals, and show that this complex is always a direct summand of the minimal free resolution of the sum of these ideals. We also give an example of an ideal where the generalized Scarf complex strictly contains the standard Scarf complex. Moreover, we introduce the notion of quasitransverse monomial ideals, and prove a list of results relating to Golodness, Koszul homology, and other homological properties for such ideals.
\end{abstract}

\section{Introduction}

In $1966$, Taylor \cite{taylor66} constructed a multigraded free resolution (now called the Taylor resolution) for all monomial ideals that is structurally analogous to the well-known Koszul complex. This complex has become a main theme in a long story on the characterization of the homological properties of monomial ideals, and is the basis of many common constructions and results relating to monomial ideals. 

As it has turned out, a closed form for the minimal free resolution of an arbitrary monomial ideal (a question posed by Kaplansky) is highly nontrivial. Many properties possessed by the Taylor resolution are now well-known to \emph{not} hold for minimal free resolutions of arbitrary monomial ideals. For instance, the frame of the Taylor resolution has differentials with coefficients $0$ or $\pm 1$, but Reiner and Welker \cite{reiner2001linear} constructed a monomial ideal for which no choice of basis for the minimal free resolution had coefficients $0$ or $\pm 1$. Likewise, the Taylor resolution does not depend on the characteristic of the underlying field, but there exist monomial ideals whose Betti numbers \emph{do} depend on the characteristic (see \cite[Example 12.4]{peeva2010graded}). Moreover, in \cite{bayer1998monomial}, the authors observe that the Taylor resolution is supported on a simplicial complex, but work of Velasco \cite{velasco2008minimal} has shown that there exist monomial ideals with minimal free resolution that cannot be supported even on a CW-complex.

Despite the above counterexamples, there have been many recent positive results on free resolutions of monomial ideals (obviously, we cannot name all of them here). In \cite{herzog2007generalization}, Herzog constructs a generalized Taylor complex and uses this resolution to prove that sums and intersections of ideals satisfy previously conjectured bounds for the regularity and projective dimension. Clark and Tchernev \cite{clark2019minimal} show that the minimal free resolution of any monomial ideal may be supported on a poset. Finally, Eagon, Miller, and Oordog \cite{eagon2019minimal} have recently presented an essentially complete solution to Kaplansky's question (with some outlier cases in small characteristic).

The purpose of the following paper is to extend results and constructions involving the Taylor resolution to the previously mentioned generalized Taylor resolution, introduced by Herzog. We first observe that one can truncate a particular direct summand of the generalized Taylor complex to obtain a free resolution of intersections of monomial ideals. We then introduce so-called quasitransverse families of monomial ideals, which are families of monomial ideals for which the generalized Taylor complex yields a minimal free resolution; in particular, a common theme in this paper is that many properties possessed by quasitransverse families of ideals are inherited by sums and intersections of these ideals (for instance, properties relating to Koszul homology/Golodness).

We show that the generalized Taylor complex may be given the structure of a DG-algebra that directly generalizes the well-known algebra structure on the classical Taylor resolution, and extend a result of Katth\"an related to quotients of the Taylor resolution by DG ideals. We construct a generalization of the Scarf complex and prove that the generalized Scarf complex is always a direct summand of the appropriate minimal free resolution. Moreover, we present an example where the generalized Scarf complex strictly contains the standard Scarf complex, suggesting that there exist ``quasiscarf" families of ideals that are not necessarily Scarf ideals. 

The paper is organized as follows. In Section \ref{sec:background}, we establish conventions and notation that will be used throughout the paper. This includes a brief introduction to the obstructions to algebra structures as introduced by Avramov in \cite{avramov1981obstructions}, which we will use in Section \ref{sec:obstructions}. In Section \ref{sec:genTaylor}, we recall Herzog's construction of the generalized Taylor complex, and observe that one can use a variant of the generalized Taylor complex to resolve intersections of monomial ideals. We also introduce quasitransverse monomial ideals, and show the aforementioned variant of the generalized Taylor resolution provides a minimal free resolution of the quotient defined by the intersection of quasitransverse ideals.

In Section \ref{sec:DGAonTaylor}, we prove that the generalized Taylor complex admits the structure of a multigraded DG-algebra (assuming the input data have algebra structures) that directly generalizes the well-known algebra structure on the classical Taylor complex. We use this to extend a result of Katth\"an by showing that if the minimal free resolution of a sum of squarefree monomial ideals admits the structure of a DG-algebra, then it can be obtained as a DG-quotient of the associated generalized Taylor complex. As an application, we reprove that the minimal free resolution of the quotient defined by the ideal $(x_1 x_2 , x_2 x_3, x_3x_4,x_4x_5,x_5x_6)$ does not admit the structure of an associative DG-algebra. 

In Section \ref{sec:genScarfComplex}, we introduce a generalization of the Scarf complex (introduced originally by Scarf \cite[2.8]{scarf1973computation}, and in the context of resolutions in \cite{bayer1998monomial}). The construction of this complex is directly analogous to the standard Scarf complex, but there are some additional subtleties to be conscious of in the more general setting. In particular, the generalized Scarf complex associated to a family of monomial ideals $I_1 , \dots , I_r$ is constructed as a subcomplex of the associated generalized Taylor resolution, and is always a direct summand of the minimal free resolution of $R/(I_1+ \cdots + I_r)$. We use this to define \emph{quasiscarf} families of ideals, and give an explicit example where the generalized Scarf complex strictly contains the standard Scarf complex.

In Section \ref{sec:KoszulHom}, we study the Koszul homology of sums and intersections of quasitransverse families of ideals. Using a result of Herzog and Steurich, we prove that the intersection of any quasitransverse family of ideals always defines a Golod ring. In the case that the family of quasitransverse ideals is squarefree, we compute the Koszul homology for both the sum and intersection explicitly; in particular, this allows us to compute an explicit trivial Massey operation on the associated Koszul complex. By Golod's construction, this yield a minimal free resolution of the residue field over the quotient defined by any intersection of a fmaily of quasitransverse monomial ideals.

In Section \ref{sec:obstructions}, we prove that the Avramov obstructions associated to quotients defined by intersections of quasitransverse families of monomial ideals are always trivial. We also use an algebra structure constructed by Geller in \cite{geller2021DG} to prove that the complex used for resolving intersections of monomial ideals defined in Section \ref{sec:genTaylor} admits the structure of an associative algebra structure under suitable hypotheses. Finally, after combining the above vanishing result with the Golodness previously established in Section \ref{sec:KoszulHom}, we prove that certain maps of Tor modules associated to quotients of intersections of quasitransverse families are always injective. We conclude the paper with further discussion and questions. 

\section{Background on Algebras, Golod Rings, and Avramov Obstructions}\label{sec:background}

In this section, we recall some background material that will appear throughout the paper. We first establish notation relating to resolutions and posets associated to resolutions. We then recall the definition of DG-algebras and Golod rings, and conclude with an introduction to the obstructions introduced by Avramov in \cite{avramov1981obstructions}.

Throughout the paper, all complexes will be assumed to have nontrivial terms appearing only in nonnegative homological degrees.

\begin{notation}
The notation $(F_\bullet , d_\bullet)$ will denote a complex $F_\bullet$ with differentials $d_\bullet$. When no confusion may occur, $F$ may be written, where the notation $d^F$ is understood to mean the differential of $F$ (in the appropriate homological degree). The notation $|f|$ will denote the homological degree of any $f \in F$.
\end{notation}

\begin{definition}
    Let $F$ be a free $R$-module with some fixed basis $B \subset F$. The \emph{support} of an element $f = \sum_{b \in B} c_b b \in F$ with respect to the basis $B$ is defined to be:
    $$\supp_B (f) := \{ b \in B \mid c_b \neq 0 \},$$
    where $c_b \in R$. When the basis $B$ is understood, the notation $\supp (f)$ will be used instead.
\end{definition}

\begin{definition}
    Let $F_\bullet$ be a complex with a fixed basis $B_i$ in each homological degree $i$. The associated poset $\po (F_\bullet)$ is defined as follows: given basis elements $e \in F_i$, $f \in F_{i-1}$,
    $$f \lessdot e \iff f \in \supp_{B_{i-1}} (d^F (f)),$$
    where $\lessdot$ denotes the covering relation. Extending $\lessdot$ transitively yields a partial order on basis elements appearing in all homological degrees.
\end{definition}

For a much more thorough treatment of DG-algebras and Golod rings, see \cite{avramov1998infinite}. We will only need a few definitions for our purposes.

\begin{definition}\label{def:dga}
A \emph{differential graded algebra} $(F,d)$ (DG-algebra) over a commutative Noetherian ring $R$ is a complex of finitely generated free $R$-modules with differential $d$ and with a unitary, associative multiplication $F \otimes_R F \to F$ satisfying
\begin{enumerate}[(a)]
    \item $F_i F_j \subseteq F_{i+j}$,
    \item $d_{i+j} (x_i x_j) = d_i (x_i) x_j + (-1)^i x_i d_j (x_j)$,
    \item $x_i x_j = (-1)^{ij} x_j x_i$, and
    \item $x_i^2 = 0$ if $i$ is odd,
\end{enumerate}
where $x_k \in F_k$.
\end{definition}

In the following definition, assume that $R$ is a local ring or a standard graded polynomial ring over a field $k$.

\begin{definition}
Let $A$ be a DG $R$-algebra with $H_0 (A) \cong k$. Then $A$ admits a \emph{trivial Massey operation} if for some $k$-basis $\mathcal{B} = \{ h_\lambda \}_{\lambda \in \Lambda}$, there exists a function
$$\mu : \coprod_{i=1}^\infty \cat{B}^i \to A$$
such that
\begingroup\allowdisplaybreaks
\begin{align*}
    &\mu ( h_\lambda) = z_\lambda \quad \textrm{with} \quad [z_\lambda] = h_\lambda, \ \textrm{and} \\
    &d \mu (h_{\lambda_1} , \dots , h_{\lambda_p} ) = \sum_{j=1}^{p-1} \overline{\mu (h_{\lambda_1} , \dots , h_{\lambda_j})} \mu (h_{\lambda_{j+1}} , \dots , h_{\lambda_p}). 
\end{align*}
\endgroup
\end{definition}

Observe that taking $p=2$ in the above definition yields that $H_{\geq 1} (A)^2 = 0$, so the induced algebra structure on $H(A)$ is trivial for a DG-algebra admitting a trivial Massey operation.

\begin{definition}\label{def:Golod}
Let $(R,\m)$ be a local ring and let $K^R$ denote the Koszul complex on the generators of $\m$. If $K^R$ admits a trivial Massey operation $\mu$, then $R$ is called a \emph{Golod ring}. An ideal $I$ in some ring $Q$ will be called \emph{Golod} if the quotient $Q/I$ is Golod.
\end{definition}

Next, we introduce the Avramov obstructions. In all of the following statements, one can instead replace $R$ with a standard graded polynomial ring over a field $k$ and work in the homogeneous setting.

\begin{definition}\label{def:avramovobstructions}
Let $(R, \m ,k )$ denote a local ring and $f : R \to S$ a morphism of rings. Let $\tor_+^R (S , k)$ denote the subalgebra of $\tor_\bullet^R (S , k)$ generated in positive homological degree. For any $S$-module $M$, there exists a map of graded vector spaces:
$$\frac{\tor_\bullet^R (M , k)}{\tor_+^R (S,k) \cdot \tor_\bullet^R (M , k)} \to \tor_\bullet^S (M , k).$$
The kernel of this map is denoted $o^f (M)$ and is called the \emph{Avramov obstruction}.
\end{definition}
The following Theorem makes clear why $o^f (M)$ is referred to as an obstruction.

\begin{theorem}[\cite{avramov1981obstructions}, Theorem 1.2]\label{thm:obstructions}
Let $(R, \m ,k )$ denote a local ring and $f : R \to S$ a morphism of rings. Assume that the minimal $R$-free resolution $F_\bullet$ of $S$ admits the structure of an associative DG-algebra. If $o^f (M) \neq 0$, then no DG $F_\bullet$-module structure exists on the minimal $R$-free resolution of $M$.
\end{theorem}

In the case that $S = R/J$ and $J$ is generated by a regular sequence, the Avramov obstructions break into ``graded" pieces:
$$o_i^f (M) := \ker \Big( \frac{\tor_i^R (M,k)}{\tor_1^R (S,k) \cdot \tor_{i-1}^R (M , k)} \to \tor_i^S (M , k) \Big).$$
Observe that this gives a method for detecting whether a minimal free resolution admits the structure of an associative DG-algebra:
\begin{prop}\label{prop:vanishofObs}
Let $(R , \m , k)$ denote a regular local ring. Suppose the minimal free resolution $F_\bullet$ of $R/I$ admits the structure of an associative DG-algebra. Then for any complete intersection $\mfa \subseteq I$, $o_i^f (R/I) = 0$ for all $i >0$, where $f : R \to R/\mfa$ is the natural quotient map.
\end{prop}
Proposition \ref{prop:vanishofObs} may tempt one to ask whether the vanishing of Avramov obstructions implies the existence of a DG algebra structure on the minimal free resolution of $R/I$. An answer in the negative was provided by Srinivasan (see \cite{srinivasan1992non}).

\section{The Generalized Taylor Complex}\label{sec:genTaylor}

In this section, we recall Herzog's construction of the generalized Taylor complex and deduce some consequences of this construction. We show that by truncating an appropriate direct summand of the Taylor resolution, one can construct a free resolution for the intersection of any family of monomial ideals. We then introduce the notion of quasitransverse families of monomial ideals, the properties of which will be studied more extensively in later sections. 

The following notation will be in play for the remainder of the paper.

\begin{notation}\label{not:ngrSetup}
Let $R = k[x_1 , \dots , x_n]$ where $k$ is any field. Throughout the paper, a \emph{multigraded} $R$-module $M$ will refer to an $R$-module that is multigraded with respect to the \emph{fine multigrading} on $R$; that is, the multigrading induced by giving each $x_i$ multidegree $\epsilon_i$ (where $\epsilon_i \in \bbz^n$ denotes the vector with a $1$ in the $i$th spot and $0$s elesewhere). Likewise, a complex of multigraded free $R$-modules $F$ will be referred to as \emph{multigraded} if all differentials are multigraded homomorphisms.
\end{notation}

\begin{notation}
Let $F $ be a multigraded complex. For any multigraded element $f \in F $, let $m_f$ denote the multidegree of $f$. Given two monomials $p_1$ and $p_2 \in R$, use the notation
$$[ p_1 , p_2] := \lcm (p_1 , p_2) \quad \textrm{and} \quad (p_1 , p_2) := \gcd (p_1 , p_2).$$
\end{notation}

The following definition comes directly from \cite{herzog2007generalization}, where we have made some adjustments to the notation to fit our needs.

\begin{definition}\label{def:genTaylor}
Let $F$ and $G$ be two multigraded complexes of free $R$-modules; by definition,
$$d^F (f) = \sum_{a} \alpha_{a,f} \frac{m_f}{m_a} a \quad \textrm{and} \quad d^G (g) = \sum_{b} \beta_{b,g} \frac{m_g}{m_b} b,$$
for some coefficients $\alpha_{a,f}, \ \beta_{b,g} \in k$, where each summation is taken over a chosen multigraded basis for $F$ and $G$, respectively. The \emph{generalized Taylor complex} on the complexes $F$ and $G$, denoted $F*G$, is the complex with
\begingroup\allowdisplaybreaks
\begin{align*}
    (F*G)_n &:= \bigoplus_{i+j = n} F_i * G_j, \quad \textrm{and differential} \\
    d^{F*G} (f * g) &:= \sum_{a} \alpha_{a,f} \frac{m_{f*g}}{m_{a*g}} a*g + (-1)^{|f|} \sum_b \beta_{b,g} \frac{m_{f*g}}{m_{f*b}} f*b, \\
    &\textrm{where} \  m_{f*g} = [m_f , m_g]. 
\end{align*}
\endgroup
\end{definition}

Notice that the generalized Taylor complex is always a direct summand of the classical Taylor complex on the respective monomial ideals. By the following result of Herzog, we see that the generalized Taylor complex is generally a closer approximation to the minimal free resolution than the classical Taylor complex.

\begin{theorem}[{\cite{herzog2007generalization}}]\label{thm:genTaylorRes}
If $F $ and $G $ are multigraded free resolutions of $R/I$ and $R/J$, where $I$ and $J$ are monomial ideals, then $F*G$ is a free resolution of $R/(I+J)$.
\end{theorem}

Whenever referring to a multigraded free resolution, we will tacitly assume that a multigraded basis has been chosen in each homological degree. The following observation is a simple description of the resolution poset associated to a generalized Taylor complex in terms of the original complexes:

\begin{obs}
Let $F $ and $G $ be multigraded free resolutions of $R/I$ and $R/J$, where $I$ and $J$ are monomial ideals. Then the poset $\po (F * G)$ is precisely the product poset $\po (F) \times \po (G)$, where each $(f,g) \in \po (F * G)$ is given multigrading $m_{f * g}$. 
\end{obs}

Recall that two ideals are \emph{transverse} if $I \cap J = IJ$. In this situation, if $F $ and $G $ are minimal free resolutions of $R/I$ and $R/J$, respectively, then the minimal free resolution of $R/(I+J)$ may be obtained as the tensor product complex $F \otimes G $. This motivates the following definition:
\begin{definition}\label{def:qTransverse}
Let $I_1 , \dots , I_r$ be a family of monomial ideals and $F^i$ a minimal free resolution of $R/ I_i$ for each $1 \leq i \leq r$. Then the family $I_1 , \dots , I_r$ is a \emph{quasitransverse} family of monomial ideals if
$$F^1 * \cdots * F^r$$
is the minimal free resolution of $R/ (I_1 + \cdots + I_r)$. 
\end{definition}

Quasitransverseness is much less restrictive than transverseness for monomial ideals; for instance, two monomials $r_1$ and $r_2$ generate transverse ideals if and only if $r_1$ and $r_2$ are coprime. On the other hand, $r_1$ and $r_2$ generate quasitransverse ideals as long as neither divides the other. One easy consequence of the definition of quasitransverse is the following formulation in terms of the associated resolution posets:

\begin{obs}
Let $I$ and $J$ be monomial ideals with minimal free resolutions $F$ and $G$ of $R/I$ and $R/J$, respectively. Then $I$ and $J$ are quasitransverse if and only if for all $(f,g) \lessdot (f',g')$ in $\po (F) \times \po (G)$, one has $\deg m_{f * g} < \deg m_{f' * g'}$. 
\end{obs}

Implicit in Herzog's paper \cite{herzog2007generalization} is the fact that one can augment a generalized Taylor resolution to obtain a free resolution of the intersection of any collection of monomial ideals; we make this observation explicit in Proposition \ref{prop:resOfInt}; first, we need to introduce some notation:

\begin{notation}
Let $(F ,d )$ denote a complex. The notation $(F_{\geq n}, d_{\geq n} )$ will denote the complex with
$$\big( F_{\geq n} \big)_i := \begin{cases}
F_i & \textrm{if} \ i \geq n \\
0 & \textrm{if} \ i < n \\
\end{cases}$$
$$\big( d_{\geq n} \big)_i := \begin{cases}
d_i & \textrm{if} \ i > n \\
0 & \textrm{if} \ i \leq n \\
\end{cases}$$
\end{notation}

\begin{definition}\label{def:resOfInt}
Let $F $ and $G $ be multigraded free resolutions of $R/I$ and $R/J$, respectively, where $I$ and $J$ are both monomial ideals. Then $F\dstar G$ is defined to be the complex induced by the differentials
$$d^{F\dstar G}_n := \begin{cases}
d^F_1 * d^G_1 & \textrm{if} \ n=1 \\
d^{F_{\geq 1} * G_{\geq 1} }_{n+1} & \textrm{otherwise}, \\
\end{cases}$$
where the notation $d^F_1 * d^G_1$ denotes the map
$$d^F_1 * d^G_1 (f_1 * g_1) := m_{f_1*g_1}.$$
Likewise, the notation $\st (F,G)$ will denote the \emph{star product}; that is, the complex induced by the differentials
$$d^{\st (F,G)}_n := \begin{cases}
d^F_1 \otimes d^G_1 & \textrm{if} \ n=1 \\
d^{F_{\geq 1} \otimes G_{\geq 1} }_{n+1} & \textrm{otherwise}. \\
\end{cases}$$
\end{definition}

\begin{remark}
In \cite{geller2021minimal} and \cite{vandebogert2020vanishing}, the the authors use the notation $F * G$ to denote the complex $\st (F,G)$ of Definition \ref{def:resOfInt}. Since the former notation conflicts with the generalized Taylor complex notation, the latter notation has been chosen to differentiate between the star product and the generalized Taylor complex.
\end{remark}

Observe that $F\dstar G$ is indeed a complex since upon localizing at the variables of $R$, $F\dstar G$ is isomorphic to $\st (F,G)$. The following proposition illustrates the relationship between the complex of Definition \ref{def:resOfInt} and free resolutions of intersections of monomial ideals.
\begin{prop}\label{prop:resOfInt}
Let $(R, \m , k)$ denote a standard graded polynomial ring over a field. Let $I$ and $J$ be monomial ideals and let $F $, $G $ denote free resolutions of $R/I$ and $R/J$, respectively. Then $F \dstar G $ is a free resolution of $R/I \cap J$. If $I$ and $J$ are quasitransverse and $F$ and $G$ are minimal, then $F \dstar G$ is minimal.
\end{prop}

\begin{proof}
There is a short exact sequence of complexes
$$0 \to F_{\geq 1} \oplus G_{\geq 1} \to (F * G)_{\geq 1} \to (F\dstar G)_{\geq 1}[-1] \to 0,$$
so by the long exact sequence of homology, $H_i (F \dstar G) = 0$ for $i \geq 2$. For $i=1$, there is a short exact sequence
$$0 \to H_1 ((F\dstar G)_{\geq 1}) \to H_1 (F_{\geq 1}) \oplus H_1 (G_{\geq 1}) \to H_1 ((F * G)_{\geq 1}) \to 0,$$
and the connecting morphism is computed as the map 
$$[f_1 * g_1] \mapsto \Big( -\frac{m_{f_1*g_1}}{m_{f_1}} [f_1] , \frac{m_{f_1*g_1}}{m_{g_1}} [g_1] \Big)$$ 
(where $[-]$ denotes homology class). Identifying the homology appearing in the latter two terms of the short exact sequence, this implies that $H_1 (F\dstar G_{\geq 1}) \cong I \cap J$ via the map $[f_1 \dstar g_1] \mapsto m_{f_1*g_1}$. Thus, augmenting $(F\dstar G)_{\geq 1}$ by the map $d_1^F * d^G_1$ remains acyclic. For the last statement, observe that since the differentials of $F \dstar G$ are built from the differentials of $F_{\geq 1} * G_{\geq 1}$, the former differentials will be minimal if the latter differentials are.
\end{proof}

\begin{remark}
By iterating Proposition \ref{prop:resOfInt}, one finds that if $I_1 , \dots , I_r$ is a family of monomial ideals with $F^i$ a free resolution of $R/I_i$ for each $1 \leq i \leq n$, then a free resolution of $R/I_1 \cap \cdots \cap I_r$ is obtained as
$$F^1 \dstar F^2 \dstar \cdots \dstar F^r .$$
\end{remark}

The following corollary shows that there exists a variant of the Taylor complex for resolving intersections of monomial ideals; we will prove in Section \ref{sec:obstructions} that this complex admits the structure of an associative DG-algebra.

\begin{cor}
Let $I_1 , \dots , I_r$ be a family of monomial ideals and let $T^i$ denote the Taylor complex on $I_i$ for each $1 \leq i \leq r$. Then $T^1 \dstar \cdots \dstar T^r$ is a free resolution of $R/I_1 \cap \cdots \cap I_r$.
\end{cor}

\begin{remark}
One advantage to using $T^1 \dstar \cdots \dstar T^r$ as a free resolution of $R / I_1 \cap \cdots \cap I_r$ is that it is built using resolutions of each $R/I_i$; this means that one can use information about each $I_i$ separately to deduce information about the intersection.
\end{remark}

As a final corollary, we observe that cellularity of minimal free resolutions is preserved for sums and intersections of quasitransverse ideals.

\begin{cor}
Let $I$ and $J$ be quasitransverse monomial ideals. If the minimal free resolutions of $R/I$ and $R/J$ can be supported on a (polyhedral) cell complex, then the minimal free resolution of both $R/(I+J)$ and $R/ I \cap J$ can be supported on a (polyhedral) cell complex.
\end{cor}

\section{Algebra Structure on the Generalized Taylor Complex}\label{sec:DGAonTaylor}

In this section, we construct an explicit DG-algebra structure on the generalized Taylor complex $F^1 * \cdots * F^r$, assuming that each $F^i$ is a DG-algebra. This recovers the algebra structure on the classical Taylor resolution by choosing each $F^i$ to be a length $1$ complex. We then obtain a generalization of a result due to Katth\"an and use this to reprove a well-known example of Avramov does not have minimal free resolution admitting the structure of an associative DG-algebra.

\begin{setup}
Let $I_1 , \dots , I_r$ be a collection of monomial ideals. Let $F^i$ be a multigraded DG-algebra free resolution of $R/I_i$, for each $1 \leq i \leq n$. The product on each $F^i $ takes the following form:
\begingroup\allowdisplaybreaks
\begin{align*}
    &f_i \cdot_F f_i' = \sum_{a_i} \alpha^i_{a_i,f_i,f_i'} \frac{m_{f_i} m_{f_i'}}{m_{a_i}} \cdot a_i , 
\end{align*}
\endgroup
for coefficients $\alpha^i_{a_i,f_i,f_i'} \in k$.
\end{setup}

The following general lemma will be a very useful method for showing that many of the DG-algebra products in this paper satisfy the needed axioms without many intensive computations. 

\begin{lemma}\label{lem:DGAafterLocal}
Let $\phi : F  \to G $ be a morphism of complexes of free $R$-modules and $S \subset R$ a multiplicative subset. Assume furthermore that $\phi : F  \to G $ is a morphism of (not necessarily DG) algebras. If $F $ is a DG-algebra and $S^{-1} \phi$ is an isomorphism of complexes, then $G $ is a DG-algebra.
\end{lemma}

\begin{proof}
By functoriality of localization and the assumption that $S^{-1} \phi$ is an isomorphism, one has
$$S^{-1} d^G = S^{-1} \phi \circ S^{-1} d^F \circ (S^{-1} \phi)^{-1}.$$
Since $F $ is assumed to be a DG-algebra and $\phi$ is a morphism of algebras, one uses the above equality to see that some $s \in S$ annihilates
$$d^G (g \cdot_G g' ) - d^G (g) \cdot_G g - (-1)^{|g|} g \cdot_G d^G (g').$$
Since $G $ is a complex of free $R$-modules, there are no torsion elements in any homological degree, whence the above expression is $0$. Likewise, the associativity of $\cdot_F$ implies that any term of the form $(g \cdot_G g') \cdot_G g'' - g \cdot_G (g' \cdot_G g'')$ is annihilated by some $s \in S$. By identical reasoning, the product $\cdot_G$ is associative.
\end{proof}

The following definition comes from \cite{katthan2019structure}; the fact that this definition is well-defined is proved in Lemma $1.4$ of that paper.

\begin{definition}
Let $V$ be a multigraded free $R$-module where all basis elements of $V$ have squarefree multidegree. Then every $v \in V$ may be written $v = r v'$ for a monomial $r \in S$ and an element $v'$ with squarefree multidegree; both $r$ and $v'$ are uniquely defined. 

The \emph{squarefree part} of $v$, denoted $v^{\sq}$, is defined to be the unique squarefree element $v'$ as above.
\end{definition}

The following proposition can be stated more conceptually by saying that a DG-algebra minimal free resolution of the quotient by a squarefree monomial ideal is essentially generated in homological degree $1$. 

\begin{prop}[{\cite[Proposition 3.4]{katthan2019structure}}]\label{prop:sqfreeGeneration}
Let $I \subset R$ be a squarefree monomial ideal with DG-algebra multigraded minimal free resolution $F$. Then every $f \in F$ of squarefree degree (with $|f| \geq 1$) may be written
$$f = \sum_j (g_{1,j} \cdot_F \cdots \cdot_F g_{|f|,j} )^{\sq}$$
for $g_{\ell,j} \in F_1$ for all $1 \leq \ell \leq |f|$ and $j$.
\end{prop}

The following theorem shows, assuming each complex separately has an algebra structure, that the generalized Taylor complex also admits the structure of a DG-algebra.

\begin{theorem}\label{thm:DGAgenTaylor}
Let $I_1 , \dots , I_r$ be a collection of monomial ideals with multigraded DG-algebra free resolutions $F^i $ of $R/ I_i$, for each $1 \leq i \leq n$. Then the complex $F^1 * \cdots * F^r$ is a DG-algebra with product
$$(f_1 * \cdots * f_r)\cdot (f_1' * \cdots * f'_r) := \sum_{a_1 , \dots , a_r} \alpha^1_{f_1,f_1',a_1} \cdots \alpha^r_{f_r, f_r' , a_r} \frac{m_{f_1 * \cdots * f_r} m_{f_1' * \cdots * f_r'}}{m_{a_1* \cdots a_r}} a_1 * \cdots * a_r.$$
Moreover, this product is multigraded.
\end{theorem}

\begin{proof}
Recall that the morphism
\begingroup\allowdisplaybreaks
\begin{align*}
    \phi : F^1 \otimes \cdots \otimes F^r &\to F^1 * \cdots * F^r, \\
    f_1 \otimes \cdots \otimes f_r &\mapsto (m_{f_1} , \dots , m_{f_r}) f_1 * \cdots * f_r
\end{align*}
\endgroup
is a morphism of complexes. If $S$ denotes the multiplicative subset generated by all powers of the variables, then $S^{-1} \phi$ is an isomorphism of complexes. In view of Lemma \ref{lem:DGAafterLocal}, it suffices to show that $\phi$ is a morphism of algebras, where $F^1 \otimes \cdots \otimes F^r$ is endowed with the natural tensor product DG-algebra structure. By induction (and for ease of notation), it suffices only to consider the case $n=2$; one computes:
\begingroup\allowdisplaybreaks
\begin{align*}
    &\phi  (f_1 \otimes f_2)  \cdot_{F^1*F^2} \phi(f_1' \otimes  f_2') \\
   =& (m_{f_1} , m_{f_2}) \cdot (m_{f_1'} , m_{f_2'}) (f_1 * f_2) \cdot_{F^1*F^2} (f_1' * f_2') \\
    =&(m_{f_1} , m_{f_2}) \cdot (m_{f_1'} , m_{f_2'})  \sum_{a_1 , a_2} \alpha^1_{f_1 , f_1' , a_1} \alpha^2_{f_2 , f_2' , a_2} \frac{m_{f_1*f_2} m_{f_1'*f_2'}}{ m_{a_1*a_2}}  a_1 * a_2  \\
    =& \sum_{a_1 , a_2} \alpha^1_{f_1 , f_1' , a_1} \alpha^2_{f_2 , f_2' , a_2} \frac{m_{f_1}m_{f_1'}}{m_{a_1}} \cdot \frac{m_{f_2} m_{f_2'}}{a_2}  (m_{a_1} , m_{a_2}) a_1 * a_2  \\
    =& \phi ( f_1 \cdot_{F^1} f_1' \otimes f_2 \cdot_{F^2} f_2') \\
    =& \phi \big( (f_1 \otimes f_2) \cdot_{F^1 \otimes F^2} (f_1' \otimes f_2') \big).
\end{align*}
\endgroup
The fact that this product is multigraded is clear by construction.
\end{proof}

\begin{remark}
Observe that if $m_1 , \dots , m_r$ is a collection of monomials and $F_i = 0 \to R \xrightarrow{m_i} R$, then $T  := F_1 * \cdots * F_r$ is isomorphic to the Taylor complex on the ideal $(m_1 , \dots , m_r)$. Moreover, one can verify that the product endowed by Theorem \ref{thm:DGAgenTaylor} is precisely the standard DG-algebra structure on the Taylor complex.
\end{remark}

The next proposition is a direct generalization of \cite[Theorem 3.6]{katthan2019structure}, and is proved in an analogous fashion using the DG-algebra of Theorem \ref{thm:DGAgenTaylor}.

\begin{prop}\label{prop:DGquotient}
Let $I_1 , \dots , I_r$ be a family of squarefree monomial ideals such that $R / I_i$ admits a DG-algebra minimal free resolution $F^i $ for each $1 \leq i \leq r$. Assume that the minimal free resolution $F $ of $R/(I_1 + \cdots + I_r)$ is a multigraded DG-algebra. Then there exists a DG-ideal $J \subset T  :=  F^1* \cdots * F^r$ such that $F \cong  T / J$ as multigraded DG-algebras.
\end{prop}

\begin{proof}
Let $S  := F^1 * \cdots * F^r$. Using the definition of the product in Theorem \ref{thm:DGAgenTaylor}, observe that for any choice of $f_i \in F^i$, one has
$$f_1 * \cdots * f_r = \frac{f_1 \cdot_S \cdots \cdot_S f_r}{(m_{f_1} , \dots , m_{f_r})} = ( f_1 \cdot_S \cdots \cdot_S f_r)^{\sq},$$
where each $f_i$ is identified with its image under the natural inclusion $F^i \subset S$. By Proposition \ref{prop:sqfreeGeneration}, one can write
$$f_i = \sum_j (g_{1,j}^i \cdot_{F^i} \cdots \cdot_{F^i} g_{|f_i|,j}^i )^{\sq},$$
where each $g_{k,j}^i$ is an element in homological degree $1$. It thus suffices to prove the statement for
$$f_i = (g_1^i \cdot_{F^i} \cdots \cdot_{F^i} g_{|f_i|}^i )^{\sq}.$$
To this end, define $\psi$ in homological degree $1$ via $\psi (g_1^i) := g_1^i$, interpreted as an element of the algebraic Scarf complex. Using this, let $\psi (f_i) := (\psi (g_1^i) \cdot_{F} \cdots \cdots_{F} \psi(g_{|f_i|}^i) )^{\sq}$ and
\begingroup\allowdisplaybreaks
\begin{align*}
    \psi (f_1 * \cdots * f_r) &:= \frac{\psi (f_1) \cdot_{F} \cdots \cdot_F \psi (f_r)}{(m_{f_1} , \dots , m_{f_r})} \\
    &= \big( \psi (f_1) \cdot_{F} \cdots \cdot_F \psi (f_r) \big)^{\sq}.
\end{align*}
\endgroup
By construction, $\psi$ is a morphism of algebras; moreover, since $\psi$ is surjective in homological degree $1$, one deduces from Proposition \ref{prop:sqfreeGeneration} that $\psi : S \to F$ is a surjection.
\end{proof}

We conclude this section with a quick application of Proposition \ref{prop:DGquotient}.

\begin{example}[See also {\cite[Example 3.7]{katthan2019structure}}]\label{ex:DGAexample}
Let $R = k[x_1 , \dots , x_6]$ and $I = (a,b,c,d,e)$ with $a = x_1 x_2$, $b = x_2 x_3$, $c = x_3 x_4$, $d = x_4 x_5$, and $e = x_5 x_6$. This is a variant of a well-known ideal defining a quotient ring with no associative DG-algebra minimal free resolution (see \cite[Example I of 2.2]{avramov1981obstructions}).

Observe that $I_1 := (a,b,c)$ and $I_2 := (d,e)$ are both minimally resolved by their corresponding algebraic Scarf complexes, denoted $S^1 $ and $S^2 $, respectively. Notice that in the Scarf complex, $g_a \cdot_{S^1} g_c = x_1 g_{b,c} + x_4 g_{a,b}$; using the algebra structure provided by Theorem \ref{thm:DGAgenTaylor},
$$(g_a * 1) \cdot_{S^1 * S^2} (g_c * g_{d,e} ) = g_{a,b} * g_{d,e} + x_4 g_{b,c} * g_{d,e}.$$
Assume that there exists a DG-ideal $J \subset S^1 * S^2$ such that  $S^1 * S^2 / J$ is the minimal free resolution of $R/I$. Notice that $g_c * g_{d,e}$ is the unique element of multidegree $x_3 x_4 x_5 x_6$ in $S^1 * S^2$. Since $\beta_{3, x_3 x_4 x_5 x_6} (R/ I ) = 0$, this means $g_c * g_{d,e}  \in J$; since $J$ is an ideal, this implies that $g_{a,b} * g_{d,e} + x_4 g_{b,c} * g_{d,e} \in J$. But $g_{a,b} * g_{d,e}$ is the unique element of $S^1 * S^2$ of multidegree $x_1 x_2 x_3 x_4 x_5 x_6$, so that after quotienting by $J$ one must have $\beta_{4 , x_1 x_2 x_3 x_4 x_5 x_6} (R/I) = 0$; however, direct computation with Macaulay2 shows that $\beta_{4 , x_1 x_2 x_3 x_4 x_5 x_6} (R/I) = 1$, a contradiction.
\end{example}

\section{A Generalization of the Scarf Complex}\label{sec:genScarfComplex}

In this section, we introduce a generalization of the Scarf complex; the Scarf complex was originally introduced by Scarf in \cite{scarf1973computation}. The first use of the Scarf complex for purposes of deducing information about resolutions of monomial ideals was in work by Bayer, Peeva, and Sturmfels \cite{bayer1998monomial}. Recall that the standard Scarf complex is constructed by restricting the Taylor complex to the subcomplex generated by all basis elements with \emph{unique} multidegree. We construct the generalized Scarf complex in a similar fashion, and prove that the generalized Scarf complex is always a direct summand of the minimal free resolution of the sum of the appropriate ideals. 

In the following, recall that the notation $G(I)$ for a monomial ideal denotes the unique minimal generating set of $I$ consisting of monic monomials.

\begin{definition}\label{def:genScarf}
Let $I_1 , \dots , I_r$ be a collection of monomial ideals with $G(I_i) \cap G(I_j) = \varnothing$ for $i \neq j$ and $F^i $ a minimal free resolution of $R / I_i$ for each $1 \leq i \leq r$. The \emph{generalized Scarf complex} $S $ on the family $I_1 , \dots , I_r$ is defined to be the subcomplex of $F  := F^1 * \cdot * F^r$ defined inductively via:
\begin{enumerate}
    \item $S_0 = R$ and $S_1 = F_1$, and
    \item $S_i$ is free on the multigraded basis elements $f_1 * \cdots * f_r \in F_i$ satisfying:
    \begin{enumerate}
        \item $d^F (f_1 * \cdots * f_r) \in S_{i-1}$, and
        \item $f_1* \cdots * f_r$ is the unique basis element of $F $ with multidegree $m_{f_1 * \cdots * f_r}$,
    \end{enumerate}
    for $i \geq 2$.
\end{enumerate}
\end{definition}

\begin{obs}
Let $I = (m_1 , \dots , m_r)$ be a monomial ideal. Then the standard Scarf complex of $I$ is obtained as the generalized Scarf complex on the family of ideals $(m_1) , \dots , (m_r)$.
\end{obs}

In the case of the standard Scarf complex, restricting to the generators of the Taylor complex with unique multidegrees yields a subcomplex; this means it is unnecessary to check condition $2 (a)$ in Definition \ref{def:genScarf}. In general, restricting to the unique multidegrees in the generalized Taylor complex does \emph{not} yield a subcomplex, so imposing $2(a)$ is necessary (see Example \ref{ex:qScarfEx}). 

\begin{prop}
The generalized Scarf complex $S $ of Definition \ref{def:genScarf} satisfies:
\begin{enumerate}
    \item $S $ is a subcomplex of $F^1 * \cdots * F^r$ satisfying $d^S (S_i) \subseteq \m S_{i-1}$ for all $i$, and
    \item $S $ is a direct summand of the minimal free resolution of $I_1 + \cdots + I_r$.
\end{enumerate}
\end{prop}

\begin{proof}
Observe that $(1)$ is immediate from Definition \ref{def:genScarf}, stated explicitly for convenience. For $(2)$, the generalized Taylor complex splits as a direct sum of the minimal free resolution of $I_1 + \cdots + I_r$ and trivial complexes of the form $0 \to R ( -m) \to R(-m) \to 0$, where $m$ is some multidegree of $F^1 * \cdots * F^r$. Since each multidegree of the generalized Scarf complex is unique, it must be contained in the minimal free resolution. 
\end{proof}

In the following definition, we will tacitly assume that a multigraded minimal free resolution has been chosen for the quotient defined by each of the respective monomial ideals, since the generalized Scarf complex of Definition \ref{def:genScarf} depends on the choice of minimal free resolutions; recall that an ideal is \emph{scarf} if the associated algebraic Scarf complex is a minimal free resolution.

\begin{definition}
A family of monomial ideals $I_1 , \dots , I_r$ is \emph{quasiscarf} if the generalized Scarf complex is a minimal free resolution of $R/ (I_1 + \cdots + I_r)$. 
\end{definition}

\begin{obs}
Let $I = (m_1 , \dots , m_r)$ be a monomial ideal. Then $I$ is a scarf ideal if and only if the family of ideals $(m_1) , \dots , (m_r)$ is quasiscarf.
\end{obs}

The following example serves to illustrate as an explicit example of an ideal whose minimal free resoution is more closely approximated by choosing a generalized Scarf complex instead of just the standard Scarf complex.

\begin{example}\label{ex:qScarfEx}
Let $R = k[x_1 , \dots ,x_6]$ and let $I$ be as in Example \ref{ex:DGAexample}. Let $I_1 = (a)$ and $I_2 = (b,c,d,e)$; let $F^1$ and $F^2$ be the multigraded minimal free resolutions of $R/I_1$ and $R/I_2$, respectively. Then, by only considering the unique multidegrees, one finds that there are $18$ unique multidegrees in $F^1 * F^2$ (compared with $15$ in the standard Scarf complex); however, restricting to these multidegrees does \emph{not} form a subcomplex of $F^1 * F^2$. If one restricts to the unique multidegrees that do form a subcomplex, there is one additional multidegree that is not included in the Scarf complex: the multidegree $x_2x_3x_4x_5x_6$, contributed by the term $1 * g$, where $g \in F^2_3$ is the unique multigraded basis element with multidegree $x_2 x_3 x_4 x_4 x_6$. One can compute the generalized Scarf complex associated to $F^1$ and $F^2$ as
\begingroup\allowdisplaybreaks
\begin{align*}&R^{1}\,\xleftarrow{\begin{pmatrix}
        {x}_{1}{x}_{2}&{x}_{2}{x}_{3}&{x}_{3}{x}_{4}&{x}_{4}{x}_{5}&{x}_{5}{x}_{6}\end{pmatrix}}\,R^{5}\\&R^{5}\,\xleftarrow{\begin{matrix}\vphantom{-{x}_{3}0-{x}_{4}{x}_{5}0-{x}_{5}{x}_{6}00}\\\vphantom{{x}_{1}-{x}_{4}000-{x}_{5}{x}_{6}0}\\\vphantom{0{x}_{2}0-{x}_{5}000}\\\vphantom{00{x}_{1}{x}_{2}{x}_{3}00-{x}_{6}}\\\vphantom{0000{x}_{1}{x}_{2}{x}_{2}{x}_{3}{x}_{4}}\end{matrix}\begin{pmatrix}
        -{x}_{3}&0&-{x}_{4}{x}_{5}&0&-{x}_{5}{x}_{6}&0&0\\
        \vphantom{\left\{2\right\}}{x}_{1}&-{x}_{4}&0&0&0&-{x}_{5}{x}_{6}&0\\
        \vphantom{\left\{2\right\}}0&{x}_{2}&0&-{x}_{5}&0&0&0\\
        \vphantom{\left\{2\right\}}0&0&{x}_{1}{x}_{2}&{x}_{3}&0&0&-{x}_{6}\\
        \vphantom{\left\{2\right\}}0&0&0&0&{x}_{1}{x}_{2}&{x}_{2}{x}_{3}&{x}_{4}\end{pmatrix}}\,R^{7}\\&R^{7}\,\xleftarrow{\begin{matrix}\vphantom{{x}_{5}{x}_{6}0}\\\vphantom{00}\\\vphantom{0{x}_{6}}\\\vphantom{00}\\\vphantom{-{x}_{3}-{x}_{4}}\\\vphantom{{x}_{1}0}\\\vphantom{0{x}_{1}{x}_{2}}\end{matrix}\begin{pmatrix}
        {x}_{5}{x}_{6}& 0 & 0 \\
        0&0 & -x_5 x_6\\
        0&{x}_{6} & 0\\
        0&0 & -x_2 x_6\\
        -{x}_{3}&-{x}_{4} & 0\\
        {x}_{1}&0 & x_4\\
        0&{x}_{1}{x}_{2} & -x_2x_3 \end{pmatrix}}\,R^{3}\\&R^{3}\,\xleftarrow{}\,0.\end{align*}
\endgroup
        Even though the above complex is not a minimal free resolution, it is closer to being a resolution than the standard Scarf complex.
\end{example}

Seeing that the generalized Scarf complex can be strictly larger than the standard Scarf complex, it is reasonable to assume that there are classes of ideals that are quasiscarf with respect to some splitting of the generators, even though the original ideal is not scarf. This motivates the following question:

\begin{question}
Are there easily describable quasiscarf families of ideals (distinct from scarf ideals)?
\end{question}

\section{Koszul Homology}\label{sec:KoszulHom}

In this section we investigate properties of the Koszul homology for (squarefree) families of quasitransverse ideals. We first deduce that arbitrary intersections of quasitransvere monomial ideals are always Golod. In the squarefree case, we give an explicit description of the generators of the Koszul homology for sums and intersections of quasitransverse families of monomial ideals and use these descriptions to deduce that other properties of Koszul homology will be preserved in these cases.

Throughout this section, $(R , \m)$ will denote a standard graded polynomial ring over a field. In the following statement, let $K $ denote the Koszul complex on the minimal generators of $\m$. The basis elements for $K_1$ will be denoted $e_1 , \dots , e_r$, and the notation $e_\sigma$ is shorthand for $e_{\sigma_1} \w \cdots \w e_{\sigma_\ell}$, where $\sigma = (\sigma_1 < \cdots < \sigma_\ell)$. Recall that if $M$ is an $R$-module, then the Koszul homology of $M$ is defined to be $H_\bullet (M \otimes K)$.

\begin{prop}\label{prop:golodQtrans}
Let $I$ and $J$ be quasitransverse monomial ideals. Then $I \cap J$ defines a Golod ring.
\end{prop}

\begin{proof}
Let $F $ and $G $ denote minimal free resolutions of $R/I$ and $R/J$, respectively. By work of Herzog and Steurich (see \cite{herzog1979golodideale}), it suffices to show that the induced map
$$H_\bullet (R/I) \oplus H_\bullet (R/J) \to H_\bullet (R/I+J)$$
is an injection. The above map on homology can be identified with the induced morphism of complexes
$$F  \otimes k \oplus G  \otimes k \to F * G \otimes k.$$
This map is evidently an injection, however, since it can be identified as the inclusion of the direct summands $F $ and $G $ into $F * G $; tensoring with $k$ preserves injectivity.
\end{proof}

\begin{notation}
Let $I$ be an ideal. A cycle in the Koszul algebra $K  \otimes R/I$ with homological degree $i$ will be denoted by $z_i^I$. The class of a cycle $z_i^I$ in the homology algebra will be denoted $[ z_i^I ]_{R/I \otimes K}$, or just $[z_i^I]$ when no confusion can occur.
\end{notation}

The following lemma shows that quasitransverse ideals satisfy a ``rescaled" version of the K\"unneth formula (one \emph{could} call it a quasik\"unneth formula, but we will not).

\begin{lemma}\label{lem:theKoszulHom}
Let $I$ and $J$ be quasitransverse squarefree monomial ideals. Then the map
\begingroup\allowdisplaybreaks
\begin{align*}
    \bigoplus_{i+j = r} H_i (R/I) \otimes_k H_j (R/J) &\to H_r (R/(I+J)) \\
    [z_i^I] \otimes [z_j^J] &\mapsto \Big[ \frac{z_i^I \w z_j^J}{( m_{z_i^I},m_{z_j^J})} \Big] \\
\end{align*}
\endgroup
is an isomorphism of vector spaces.
\end{lemma}

\begin{proof}
First, it is necessary to show that $\frac{z_i^I \w z_j^J}{(m_{z_i^I},m_{z_j^J})}$ is well-defined. Write $z_i^I = \sum r_\sigma e_\sigma$, $z_j^J = \sum r_\tau e_\tau$; using the squarefree assumption, one finds
$$\frac{z_i^I \w z_j^J}{(m_{z_i^I},m_{z_j^J})} = \sum_{\substack{\sigma , \tau \\
\sigma \cap \tau = \varnothing \\}} [r_ \sigma , r_\tau ] e_\sigma \w e_\tau,$$
which is a well-defined multigraded cycle in $K  \otimes R/(I+J)$. To complete the proof, it suffices to show that all elements of the form $\frac{z_i^I \w z_j^J}{(m_{z_i^I},m_{z_j^J})}$ descend to a basis in homology, assuming $z_i^I$ and $z_j^J$ represent basis elements in their respective homology algebras.

Recall that the isomorphism $F*G \otimes k \cong H_\bullet (R/(I+J) \otimes K )$ is also an isomorphism of algebras. Let $f$ and $g$ denote representatives for the images of $z_i^I$ and $z_j^J$ under the isomorphisms $H_\bullet (R/I \otimes K) \cong F \otimes k$ and $H_\bullet (R/ J \otimes K) \cong G \otimes k$, respectively. Since $f*g = \frac{(f*1) \cdot (1*g)}{(m_f , m_g)}$ by definition of the product of Theorem \ref{thm:DGAgenTaylor}, it follows that the image of $\overline{f*g}$ under the above isomorphism is precisely $\frac{z_i^I \w z_j^J}{(m_f , m_g)}$; since $f *g$ represents a basis element in $F*G \otimes k$, the result follows.


\end{proof}

\begin{cor}\label{cor:KoszulforInt}
Let $I$ and $J$ be quasitransverse squarefree monomial ideals. Then the map
\begingroup\allowdisplaybreaks
\begin{align*}
    \bigoplus_{\substack{i+j = r+1 \\
    i,j \geq 1 \\}} H_i (R/I) \otimes_k H_j (R/J) &\to H_r (R/I \cap J)) \\
    [z_i^I] \otimes [z_j^J] &\mapsto \Big[ \frac{d(z_i^I) \w z_j^J}{(m_{z_i^I},m_{z_j^J})} \Big] \\
\end{align*}
\endgroup
is an isomorphism of vector spaces for all $i \geq 1$.
\end{cor}

\begin{proof}
Let $K$ denote the Koszul complex resolving the residue field $k$. We will instead compute the Koszul homology $H_i (I\cap J)$ and use the isomorphism $H_{i+1} (R/I\cap J) \cong H_i (I \cap J)$ given by sending a cycle $[z]_{R/I \cap J \otimes K} \mapsto [d(z)]_{I \cap J \otimes K}$ (where $d$ denotes the Koszul differential). By Lemma \ref{lem:theKoszulHom}, a basis for the Koszul homology of $I+J$ is represented by all elements of the form
$$\frac{d(z_i^I) \w z_j^J}{(m_{z_i^I} , m_{z_j^J})} + (-1)^i \frac{z_i^I \w d(z_j^J)}{(m_{z_i^I} , m_{z_j^J})}.$$
Using the short exact sequence
$$0 \to I \cap J \to I \oplus J \to I + J \to 0,$$
the connecting morphism is computed as the map
$$\Big[ \frac{d(z_i^I \w z_j^J)}{(m_{z_i^I} , m_{z_j^J})} \Big]_{(I+J) \otimes K} \mapsto \begin{cases}
\Big[ \frac{d(z_i^I) \w d(z_j^J)}{(m_{z_i^I} , m_{z_j^J})} \Big]_{I \cap J \otimes K} & \textrm{if} \ 1 \leq i \leq n-1 \\
0 & \textrm{otherwise}. \\
\end{cases}$$
By the proof of Proposition \ref{prop:golodQtrans}, the map $H_r(I ) \oplus H_r (J) \to H_r (I+J)$ is an injection. This means that the long exact sequence of homology splits into short exact sequences
$$0 \to H_r (I) \oplus H_r(J) \to H_r (I+J) \to H_{n-1} (I \cap J) \to 0.$$
This immediately implies that $H_{n-1} (I \cap J)$ has basis given by all cycles of the form $\Big[ \frac{d(z_i^I) \w d(z_j^J)}{(m_{z_i^I} , m_{z_j^J})} \Big]_{I \cap J \otimes K}$; finally, to conclude the proof observe that the Leibniz rule implies
$$\Big[ \frac{d(d(z_i^I) \w z_j^J)}{(m_{z_i^I} , m_{z_j^J})} \Big]_{I \cap J \otimes K} = (-1)^{i-1} \Big[ \frac{d(z_i^I) \w d(z_j^J)}{(m_{z_i^I} , m_{z_j^J})} \Big]_{I \cap J \otimes K}.$$
\end{proof}

\begin{definition}
Let $I$ be a monomial ideal. Then the Koszul homology $H_\bullet (R/I)$ has \emph{expected form} if it is generated by all elements of the form $\overline{u} e_\sigma$, where $u \in I : (x_{\sigma_1} , \dots , x_{\sigma_\ell})$ (and $|\sigma| = \ell$). 
\end{definition}

\begin{remark}
By \cite{peeva19960} and \cite{gasharov2002resolutions}, the Koszul homology of any $\mfa$-stable monomial ideal has expected form. Likewise, a recent result of Dao and De Stefani (see \cite{dao2020monomial}) shows that any monomial ideal in $k[x,y,z]$ has Koszul homology of expected form.
\end{remark}

\begin{example}[{\cite[Example 2.9]{dao2020monomial}}]
The Koszul homology of $R/I$, where $$I = (xz , xw , yz , yw , x^2 , y^2 , z^2 , w^2 ) \subset k[x,y,z,w],$$ does \emph{not} have expected form. In particular, $x e_{yzw} - y e_{xzw}$ represents a basis element that can not be rewritten in terms of ``expected" generators.
\end{example}

Using the Lemma \ref{lem:theKoszulHom} and Corollary \ref{cor:KoszulforInt}, we find that having Koszul homology of expected form is preserved for both sums and intersections of quasitransverse ideals.

\begin{cor}
Let $I$ and $J$ be quasitransverse squarefree monomial ideals. If the Koszul homology of $R/I$ and $R/J$ has expected form, then the Koszul homology of both $R/(I+J)$ and $R/(I \cap J)$ has expected form.
\end{cor}

Since we have an explicit generating set for the Koszul homology of squarefree quasitransverse ideals, we will be able to construct an explicit trivial Massey operation and reprove the fact that intersections of quasitransverse ideals are Golod; this will also yield an explicit construction of the minimal free resolution of the residue field over the quotient defined by the intersection of a quasitransverse family of ideals (this follow by a result of Golod \cite{golod1962}; see also the construction \cite[Theorem 5.2.2]{avramov1998infinite}. The following proposition, proved in \cite{vandebogert2021Golod}, will allow us to easily construct the desired trivial Massey operation. 

\begin{prop}\label{prop:golodnessLem}
Let $I \subset R$ be any ideal and enumerate a $k$-basis $\cat{B} = \{ [z_a ] \}_{a \in A}$ for the Koszul homology algebra $H_{\geq 1} (R/I)$. Assume that there exists a map $\nu : \{ z_a \mid a \in A \} \to R/I \otimes K $ such that
$$(*) \qquad z_a \w z_b = z_a \w d( \nu (z_b)) \quad \textrm{for all} \ a,  \ b \in A.$$
Then the map $\mu : \coprod_{i=1}^\infty \cat{B}^i \to R/I \otimes K $ defined by
$$\mu ([z_1] , \dots , [z_p] ) := z_1 \w \nu (z_2) \w \cdots \w \nu (z_p)$$
is a trivial Massey operation on $R/I \otimes K $.
\end{prop}

For ease of notation in the below statement, use the notation
$$(z_i^I \w z_j^J)^{\sq} := \frac{z_i^I \w z_j^J}{(m_{z_i^I} , m_{z_j^J})}.$$

\begin{cor}
Let $I$ and $J$ be quasitransverse monomial ideals. Let
\begingroup\allowdisplaybreaks
\begin{align*}
    &\mu : \coprod_{i =1}^\infty \cat{B}^i \to H_{\geq 1} (R/I \cap J), \\
    &\mu([(d(z_{a_1}^I) \w z_{b_1}^J)^{\sq} ], \cdots , [(d(z_{a_p}^I) \w z_{b_p}^J)^{\sq} ] )  \\
     :=& (d(z_{a_1}^I ) \w z_{b_1}^J)^{\sq} \w (z_{a_2}^I \w z_{b_2}^J)^{\sq} \w \cdots \w (z_{a_p}^I \w z_{b_p}^J)^{\sq}. \\
\end{align*}
\endgroup
Then $\mu$ is a trivial Massey operation on $R/I \cap J \otimes_R K $.
\end{cor}

\begin{proof}
Define $\nu ( (d(z_i^I) \w z_j^J)^{\sq} ) := (z_i^I \w z_j^J)^{\sq}$. Then $\nu$ satisfies the condition $(*)$ of Proposition \ref{prop:golodnessLem}, whence the result.
\end{proof}

\section{Vanishing of Avramov Obstructions for Intersections of Quasitransverse Ideals}\label{sec:obstructions}

In this section, we conduct a further study of the complex $F^1 \dstar \cdots \dstar F^r$ introduced in Definition \ref{def:resOfInt}. We use an algebra structure constructed by Geller \cite{geller2021DG} to prove that $F^1 \dstar \cdots \dstar F^r$ admits the structure of an associative DG-algebra under suitable hypotheses on the complexes $F^1 , \dots , F^r$. We then prove that, under rather mild hypotheses, the complex $F^1 \dstar \cdots \dstar F^r$ will admit a DG-module structure over the exterior algebra $\bigwedge^\bullet (F^1 \dstar \cdots \dstar F^r)_1$; this shows that the Avramov obstructions of Definition \ref{def:avramovobstructions} are always $0$ for quotients of intersections of quasitransverse families of ideals. 

We first begin by establishing the following notation to define the product of Proposition \ref{prop:DGAdstar}; much of this notation comes from \cite{geller2021DG}.

\begin{notation}
Let $\Delta$ denote a simplicial complex. For ease of notation, the vertices of $\Delta$ will be identified with their respective indices for any choice of enumeration. This allows one to order the verties with respect to the natural ordering on the integers $\bbn$. 

 If $\Omega \in \Delta$ is a face of $\Delta$, then vertices of $\Omega$ will be denoted by the corresponding lower-case Greek letter. For instance, vertices of $\Omega \in \Delta$ are denoted $\omega_i$ and vertices of $\Gamma \in \Delta$ are denoted $\gamma_i$, where $i$ denotes the position in $\Omega$ (or $\Gamma$) with respect to the order given above. 
\end{notation}

\begin{definition}
Let $P$ denote a statement that can be evaluated as \texttt{true} or \texttt{false}. The \emph{indicator function} $1_P$ corresponding to $P$ has values
$$1_P = \begin{cases}
1 & \textrm{if} \ P \ \textrm{is true} \\
0 & \textrm{if} \ P \ \textrm{is false.} \\
\end{cases}$$
\end{definition}

\begin{definition}
A multigraded complex $G $ is \emph{supported on a simplicial complex} $\Delta$ (or just \emph{simplicial}) if all basis elements of $G $ are indexed by the faces of $\Delta$ and the differentials of $G $ are induced by the topological differentials on $\Delta$. Given a basis element $g_{\Omega} \in G $ (where $\Omega \in \Delta$ and $|g_\Omega| = |\Omega|$), the differential takes the form
$$d^G (g_\Omega) = \sum_{\omega \in \Omega} \epsilon (\omega) \frac{m_{\Omega}}{m_{\Omega \backslash \omega}} g_{\Omega \backslash \omega},$$
where $\epsilon (\omega)$ is shorthand for the incidence function $\epsilon (\Omega \backslash \omega , \Omega)$. Given $G $ a simplicial complex as above, let $P_\ell : G_i \to G_{i-1}$ to be the map defined on basis vectors via
$$P_\ell (g_\Omega ) := \proj_{g_{\Omega \backslash \omega_\ell}} (d^G (g_{\Omega} )) = \epsilon (\omega_\ell) \frac{m_{\Omega}}{m_{\Omega \backslash \omega_\ell}} g_{\Omega \backslash \omega_\ell},$$
where $\omega_\ell$ is understood to be the $\ell$th element of $\Omega$.

If $G $ admits the structure of a multigraded DG-algebra, then the multiplication is \emph{simplicial} if
$$g_{\Omega} \cdot_G g_{\Gamma} = \lambda g_{\Omega \cup \Gamma}, \quad \textrm{for some} \ \lambda \in k.$$
\end{definition}

Notice that the following proposition has a quick proof only because the hard work has already been done in the construction of the DG-algebra structure on $\st (F,G)$ established under suitable hypotheses by Geller \cite{geller2021DG}; this means that we only need to invoke Lemma \ref{lem:DGAafterLocal} to deduce the result.

\begin{prop}\label{prop:DGAdstar}
Let $I$ and $J$ be monomial ideals. Let $F $ and $G $ be multigraded DG-algebra resolutions of $R/I$ and $R/J$, respectively. Assume that $G $ is a simplicial resolution supported on $\Delta$ with simplicial multiplication. Then $F \dstar G$ admits the structure of a multigraded DG-algebra with product:
\begingroup\allowdisplaybreaks
\begin{align*}
    (f * g_{\Omega} ) \cdot_{F \dstar G} (f' * g_{\Gamma}) &:= 1_{[\omega_1 \leq \gamma_1 < \omega_2]} (-1)^{(|\Omega|-1)(|f'|-1)} (f * P_1 (f_{\Omega})) \cdot_{F * G} (f' * g_{\Gamma}) \\
    &- 1_{[\omega_1 < \gamma_1]} 1_{[|f'|=1]} (f * g_{\Omega}) \cdot_{F*G} (d^F (f') * g_{\Gamma} ) \\
    & 1_{[\gamma_1 < \omega_1 < \gamma_2]} (-1)^{(|\Omega| -1)|f'|} (f * g_{\Omega} ) \cdot_{F*G} (f' * P_1 (g_\Gamma) ) \\
    &- 1_{[\gamma_1 < \omega_1]} 1_{[|f|=1]} (-1)^{|\Omega| (|f'|-1)} (d^F (f) * g_\Omega) \cdot_{F*G} (f' * g_\Gamma). \\
\end{align*}
\endgroup
\end{prop}

\begin{proof}
The proof of this theorem is immediate from Theorem $1.1$ of \cite{geller2021DG}. Observe that there is a monomorphism of complexes
\begingroup\allowdisplaybreaks
\begin{align*}
    \phi: \st (F , G) & \to F \dstar G \\
    f \otimes g &\mapsto (m_f , m_g) f * g. 
\end{align*}
\endgroup
By the proof of Theorem \ref{thm:DGAgenTaylor}, the map $\phi$ is a morphism of algebras with respect to the products $\cdot_{F \otimes G}$ and $\cdot_{F*G}$. Since $\cdot_{F \dstar G}$ is defined in terms of linear combinations of $\cdot_{F*G}$, one immediately finds that $\phi$ is a morphism of algebras when $\cdot_{\st (F,G)}$ is chosen to be the product of \cite[Theorem 1.1]{geller2021DG}. By Lemma \ref{lem:DGAafterLocal}, the result follows. 
\end{proof}

\begin{remark}
Let $I_1 , \dots , I_r$ be a collection of monomial ideals and $T^i $ the Taylor resolution on $I_i$ for each $1 \leq i \leq r$. Then, Proposition \ref{prop:DGAdstar} can be iterated to show that 
$$T^1 \dstar \cdots \dstar T^r$$
admits the structure of an associative DG-algebra. 
\end{remark}

Next, we adopt the following setup to prove that the associated Avramov obstructions are trivial for general intersections of quasitransverse monomial ideals.

\begin{setup}\label{setup:obstructionsetup}
Let $I_1 , \dots , I_r$ denote a family of monomial ideals. Let $(F^i, d^i)$ denote a free resolution of $R/I_i$ for each $i=1 , \dots , n$. Assume that for each $i$ and $j$, there exists a product $\cdot : F_1^i \otimes F_j^i \to F_{j+1}^i$ satisfying
\begin{enumerate}[(a)]
    \item $d_{j+1}^i (f_1^i \cdot f_j^i ) = d_1^i (f_1^i) f_j^i - f_1^i \cdot d_j (f_j^i)$, and
    \item $f_1^i \cdot (f_1^i \cdot f_j^i) = 0$, where $f_j^i \in F_j^i$.
\end{enumerate}
\end{setup}

\begin{remark}
If $I_1 , \dots , I_r$ is a family of monomial ideals, each of which admits a minimal free resolution with the structure of an associative DG algebra, then the hypotheses of Setup \ref{setup:obstructionsetup} are satisfied by Definition \ref{def:dga}.
\end{remark}

The following theorem is a consequence of the computation done in \cite[Theorem 5.5]{vandebogert2020vanishing}.

\begin{theorem}\label{thm:Dgmodules}
Adopt notation and hypotheses as in Setup \ref{setup:obstructionsetup}. Let $S  := F^1 \dstar \cdots \dstar F^r$. Then there exists a product
$$S_1 \otimes S_j \to S_{j+1}$$
satisfying
\begin{enumerate}[(a)]
    \item $d_{j+1}^i (f_1^i \cdot f_j^i ) = d_1^i (f_1^i) f_j^i - f_1^i \cdot d_j (f_j^i)$, and
    \item $f_1^i \cdot (f_1^i \cdot f_j^i) = 0$.
\end{enumerate}
\end{theorem}

\begin{proof}
By induction it suffices to consider the case that $I$ and $J$ are monomial ideals with $F $, $G $ free resolutions of $R/I$ and $R/J$, respectively. For ease of notation, elements of $F_i$ will denoted denoted $f_i$ and elements of $G_j$ will be denoted $g_j$. Define
$$(f_1 \otimes g_1) \cdot (f_a \otimes g_b ) := (-1)^a (d^F_1 (f_1) * g_1) \cdot_{F*G} (f_a *  g_b) +  1_{[b=1]} (f_1 * g_1) \cdot_{F*G} (f_a * d_b^G(g_b)). $$
By construction, the map $\phi : \st (F,G) \to F \dstar G$ satisfies
$$\phi \big( (f_1 \otimes g_1) \cdot_{\st (F,d)} (f_a \otimes f_b) \big) = \phi (f_1 \otimes g_1) \cdot_{F \dstar G} \phi (f_a \otimes g_b),$$
where $\cdot_{\st (F,G)}$ is the product defined in Theorem $5.5$ of \cite{vandebogert2020vanishing}. By the proof of Lemma \ref{lem:DGAafterLocal}, the result follows.
\end{proof}

\begin{cor}\label{cor:injTor}
Adopt notation and hypotheses as in Setup \ref{setup:obstructionsetup} and define $J := I_1 \cap \cdots \cap I_r$. If $\mfa \subset J$ is any complete intersection, then the induced map
$$\tor_i^R (R/J,k) \to \tor^S_i (R/J , k)$$
is injective for all $i \geq 2$, where $S = R/\mfa$. 
\end{cor}

\begin{proof}
By Proposition \ref{prop:golodQtrans}, one has
$$\tor_1^R (S , k ) \cdot \tor_{i-1} (R/J , k) = 0$$
for all $i \geq 2$. The statement of the corollary is then a rephrasing of the fact that $o_i^f (R/J) = 0$, where $f : R \to S$ is the natural quotient.
\end{proof}

\bibliographystyle{amsplain}
\bibliography{biblio}

\providecommand{\bysame}{\leavevmode\hbox to3em{\hrulefill}\thinspace}
\providecommand{\MR}{\relax\ifhmode\unskip\space\fi MR }
\providecommand{\MRhref}[2]{%
  \href{http://www.ams.org/mathscinet-getitem?mr=#1}{#2}
}
\providecommand{\href}[2]{#2}
\begin{thebibliography}{10}

\bibitem{avramov1981obstructions}
Luchezar~L Avramov, \emph{Obstructions to the existence of multiplicative
  structures on minimal free resolutions}, American Journal of Mathematics
  \textbf{103} (1981), no.~1, 1--31.

\bibitem{avramov1998infinite}
\bysame, \emph{Infinite free resolutions}, Six lectures on commutative algebra,
  Springer, 1998, pp.~1--118.

\bibitem{bayer1998monomial}
Dave Bayer, Irena Peeva, and Bernd Sturmfels, \emph{Monomial resolutions},
  Mathematical Research Letters \textbf{5} (1998), no.~1, 31--46.

\bibitem{clark2019minimal}
Timothy Clark and Alexandre Tchernev, \emph{Minimal free resolutions of
  monomial ideals and of toric rings are supported on posets}, Transactions of
  the American Mathematical Society \textbf{371} (2019), no.~6, 3995--4027.

\bibitem{dao2020monomial}
Hailong Dao and Alessandro De~Stefani, \emph{On monomial golod ideals}, Acta
  Mathematica Vietnamica (2020), 1--9.

\bibitem{eagon2019minimal}
John Eagon, Ezra Miller, and Erika Ordog, \emph{Minimal resolutions of monomial
  ideals}, arXiv preprint arXiv:1906.08837 (2019).

\bibitem{gasharov2002resolutions}
Vesselin Gasharov, Takayuki Hibi, and Irena Peeva, \emph{Resolutions of
  a-stable ideals}, Journal of Algebra \textbf{254} (2002), no.~2, 375--394.

\bibitem{geller2021DG}
Hugh Geller, \emph{Differential graded resolutions from tensor products of
  truncations}, In Preparation.

\bibitem{geller2021minimal}
\bysame, \emph{Minimal free resolutions of fiber products}, arXiv preprint
  arXiv:2104.04390 (2021).

\bibitem{golod1962}
E.~S. Golod, \emph{Homologies of some local rings}, Dokl. Akad. Nauk SSSR
  \textbf{144} (1962), 479--482. \MR{0138667}

\bibitem{herzog2007generalization}
J{\"u}rgen Herzog, \emph{A generalization of the taylor complex construction},
  Communications in Algebra \textbf{35} (2007), no.~5, 1747--1756.

\bibitem{herzog1979golodideale}
J{\"u}rgen Herzog and Manfred Steurich, \emph{Golodideale der gestalt $a \cap
  b$}, Journal of Algebra \textbf{58} (1979), no.~1, 31--36.

\bibitem{katthan2019structure}
Lukas Katth{\"a}n, \emph{The structure of dga resolutions of monomial ideals},
  Journal of Pure and Applied Algebra \textbf{223} (2019), no.~3, 1227--1245.

\bibitem{peeva19960}
Irena Peeva, \emph{0-borel fixed ideals}, Journal of Algebra \textbf{184}
  (1996), no.~3, 945--984.

\bibitem{peeva2010graded}
\bysame, \emph{Graded syzygies}, vol.~14, Springer Science \& Business Media,
  2010.

\bibitem{reiner2001linear}
Victor Reiner and Volkmar Welker, \emph{Linear syzygies of stanley-reisner
  ideals}, Mathematica Scandinavica (2001), 117--132.

\bibitem{scarf1973computation}
Herbert~E Scarf and Terje Hansen, \emph{The computation of economic
  equilibria}, no.~24, Yale University Press, 1973.

\bibitem{srinivasan1992non}
Hema Srinivasan, \emph{The non-existence of a minimal algebra resolution
  despite the vanishing of avramov obstructions}, Journal of Algebra
  \textbf{146} (1992), no.~2, 251--266.

\bibitem{taylor66}
Diana~Kahn Taylor, \emph{Ideals generated by monomials in an r-sequence}, Ph.D.
  thesis, University of Chicago, Department of Mathematics, 1966.

\bibitem{vandebogert2021Golod}
Keller VandeBogert, \emph{Products of ideals and golod rings}, In Preparation.

\bibitem{vandebogert2020vanishing}
\bysame, \emph{Vanishing of avramov obstructions for products of sequentially
  transverse ideals}, arXiv preprint arXiv:2011.11665 (2020).

\bibitem{velasco2008minimal}
Mauricio Velasco, \emph{Minimal free resolutions that are not supported by a
  cw-complex}, Journal of Algebra \textbf{319} (2008), no.~1, 102--114.

\end{thebibliography}
\addcontentsline{toc}{section}{Bibliography}

\end{document}